\documentclass[a4paper]{amsart}
\usepackage[T1]{fontenc}
\usepackage[utf8]{inputenc}
\usepackage{amssymb, amscd, mathrsfs, xifthen, xfrac}
\usepackage[colorinlistoftodos]{todonotes}
\newtheorem{thm}{Theorem}[section]
\newtheorem{obs}[thm]{Observation}
\newtheorem{lem}[thm]{Lemma}
\newtheorem{prop}[thm]{Proposition}
\theoremstyle{definition}
\newtheorem*{df}{Definition}
\newcommand{\gp}{\mathfrak{p}}
\newcommand{\gpp}{\mathfrak{p}'}
\newcommand{\gq}{\mathfrak{q}}
\newcommand{\gqq}{\mathfrak{q}'}
\newcommand{\gr}{\mathfrak{r}}
\newcommand{\EE}{\mathbb{E}}
\newcommand{\FF}{\mathbb{F}}
\newcommand{\QQ}{\mathbb{Q}}
\newcommand{\ZZ}{\mathbb{Z}}
\newcommand{\CS}{\mathcal{S}}
\newcommand{\CSS}{\mathcal{S}'}

\newcommand{\CW}{\mathcal{W}}
\newcommand{\SB}{\mathscr{B}}
\newcommand{\SD}{\mathscr{D}}
\newcommand{\SE}{\mathscr{E}}
\newcommand{\vt}{\check{t}}
\newcommand{\vT}{\check{T}}
\newcommand{\card}[1]{|#1|}

\newcommand{\onto}{\twoheadrightarrow}
\newcommand{\iso}{\xrightarrow{\raisebox{-2bp}{\smash{\tiny$\,\sim\,$}}}}
\newcommand{\units}[1]{#1^\times}
\newcommand{\squares}[1]{#1^{\times2}}
\newcommand{\sing}[1]{E_{#1}}
\newcommand{\singX}{\sing{X}}
\newcommand{\Sing}[1]{\EE_{#1}} 
\newcommand{\SingX}{\Sing{X}}
\newcommand{\SingY}{\Sing{Y}}
\newcommand{\SingZ}{\Sing{Z}}
\newcommand{\dl}[1]{\Delta_{#1}}
\newcommand{\Dl}[1]{\mathbf{\Delta}_{#1}} 
\newcommand{\st}{\mathrel{\mid}}
\DeclareMathOperator{\rank}{rk}
\newcommand{\rk}[1][]{\ifthenelse{\equal{#1}{}}{\rank_{2}}{\rank_{#1}}}
\newcommand{\Tt}[1][]{\ifthenelse{\equal{#1}{}}{(T,t)}{(T_{#1},t_{#1})}}
\newcommand{\smalleq}[2]{\bigl(T_{#1}, t_{#1}, (t_{#2}\st #2\in #1)\bigr)}
\newcommand{\preeq}[2]{\bigl(\vT, \vt, (\vt_{#2}\st #2\in #1)\bigr)}
\DeclareMathOperator{\Div}{Div}
\newcommand{\DivX}{\Div X}
\newcommand{\DivY}{\Div Y}	
\DeclareMathOperator{\Pic}{Pic}
\newcommand{\PicX}{\Pic X}
\newcommand{\PicZeroX}{\Pic^0 X}
\newcommand{\PicXS}[1][]{\Pic(X\setminus \CS#1)}
\newcommand{\PicY}{\Pic Y}	
\newcommand{\PicTY}{\Pic (TY)}	
\newcommand{\PicZ}{\Pic Z}	
\newcommand{\class}[1]{[#1]}
\DeclareMathOperator{\dv}{div}
\newcommand{\divX}{\dv_X}
\newcommand{\divY}{\dv_Y}	

\DeclareMathOperator{\lin}{span}
\DeclareMathOperator{\ord}{ord}
\newcommand{\un}[1][]{\ifthenelse{\equal{#1}{}}{K^\times}{#1^\times}}	
\newcommand{\sqgd}[1]{\sfrac{#1^\times}{\squares{#1}}}
\newcommand{\wild}[1]{\CW #1}
\newcommand{\term}[1]{\emph{#1}}
\def\clap#1{\hbox to 0pt{\hss#1\hss}}

\def\mathclap{\mathpalette\mathclapinternal}

\def\mathclapinternal#1#2{\clap{$\mathsurround=0pt#1{#2}$}}
\newwrite\refs
\openout\refs=\jobname.refs
\makeatletter
\renewcommand\@setref[3]{%
        \ifx#1\relax
                \write\refs{'#3' \thepage\space undefined}%
                \protect \G@refundefinedtrue
                \nfss@text{\reset@font\bfseries ??}%
                \@latex@warning{Reference `#3' on page \thepage\space
                                undefined}%
        \else
                \write\refs{'#3' \thepage\space
                            \expandafter\@secondoftwo#1}%
                \expandafter#2#1\null
        \fi
}
\makeatother
\author[A. Czoga{\l}a \and P. Koprowski \and B. Rothkegel]{Alfred Czoga\l a \and Przemys{\l}aw Koprowski \and Beata Rothkegel}
\address{Institute of Mathematics\\
University of Silesia\\ Bankowa 14\\ 40-007 Katowice,
Poland} \email{alfred.czogala@us.edu.pl}
\address{Institute of Mathematics\\
University of Silesia\\ Bankowa 14\\ 40-007 Katowice,
Poland} \email{przemyslaw.koprowski@us.edu.pl}
\address{Institute of Mathematics\\
University of Silesia\\ Bankowa 14\\ 40-007 Katowice,
Poland} \email{brothkegel@math.us.edu.pl}
\title{Wild sets  in global function fields}
\begin{document}

\maketitle
\begin{abstract}
Given a self-equivalence of a global function field, its wild set is the set of points where the self-equivalence fails to preserve parity of valuation. In this paper we describe structure of finite wild sets.
\end{abstract}

\section{Introduction}
In order to fully understand the theory of quadratic forms over a given base field~$K$, one considers automorphisms of its Witt ring~$WK$, i.e. the ring of similarity classes of non-degenerate quadratic forms over~$K$. When $K$ is a global field (either a number field or a function field) this boils down to investigating self-equivalences (see definition below) of~$K$. The correspondence between self-equivalences of~$K$ and automorphisms of~$WK$ (more generally between Hilbert-symbol equivalence\footnote{Hilbert-symbol equivalence was originally called \term{reciprocity equivalence}. The term Hilbert-symbol equivalence was introduced later.} and Witt equivalence) originates from works of J.~Carpenter, P.E.~Conner, R.~Litherland, R.~Perlis, K.~Szymiczek and the first author in early 1990s (see e.g. \cite{PSCL94}) and has been developed in numerous papers since then. A \term{wild set} of a given self-equivalence is the set of points, called \term{wild points}, where the self-equivalence fails to preserve parity of valuation (for a rigorous definition see below). The wild points constitute and obstruction for the associated Witt automorphism to map the Witt ring of an underlying maximal order (say the ring of algebraic integers in number theoretical case or the ring of polynomial functions in geometric case) to itself. This relation was described by the first author in \cite{Czogala01}. For these reasons it is important to understand the structure of wild sets of global fields. The problem was investigated first by T.~Palfrey in~\cite{Palfrey98}. Next, M.~Somodi in \cite{Somodi06, Somodi08} completely described the structure of wild sets of~$\QQ$ and~$\QQ[i]$. Subsequently two of the present authors obtained a partial characterization of wild sets of number fields. Namely it is proved in \cite{CR14} that if $\CS = \{\gp_1, \dotsc, \gp_n\}$ is a set of finite primes of a number field~$K$ such that $-1$ is a local square at each of these primes and the classes of~$\gp_i$, $i\leq n$ are $2$-divisible in the ideal class group of~$K$, then~$\CS$ is a wild set. This result was generalized in our earlier paper to global function fields (see Theorem~\ref{thm:rank0}), providing a complete characterization of wild sets of rank~$0$. However, it was already remarked in \cite{CKR18} that these are not all possible wild set of a global function field. The purpose of the present paper is to describe wild of arbitrary rank. In particular we show (see Theorem~\ref{thm:half_points_lin_indep}) that the number of points in a wild set must exceed its rank at least twice. Next, in Theorem~\ref{thm:rank1}, we characterize wild sets of rank~$1$. Finally, in Proposition~\ref{prop:arbitrary_2rank} we obtain sufficient conditions for sets of higher ranks to be wild sets.

In this paper $K$ is always a global function field of characteristic $\neq 2$ and $\FF_q$ is its full field of constants. We may think of~$K$ as a field of rational functions on a smooth complete curve~$X$. The points of~$X$ are identified with classes of discrete valuations on~$K$. Since we shall never explicitly refer to the generic point of~$X$, for brevity we say ``point'' to mean ``closed point''. We denote the set of closed points again by~$X$.

Given a point $\gp\in X$, $\ord_\gp:\units{K}\onto \ZZ$ is the associated normalized discrete valuation, $K_\gp$ denotes the completion of~$K$ at~$\gp$ and $K(\gp)$ is the residue field. The degree $[K(\gp):\FF_q]$ of the residue field of~$\gp$ over the field of constants is called the \term{degree} of~$\gp$ and denoted $\deg\gp$. The divisors of~$X$ are written additively. Given a divisor $\SD = \sum_{\gp\in X} n_\gp\cdot \gp$, its degree is $\deg\SD := \sum_{\gp\in X}n_\gp\cdot \deg\gp$. The class of~$\SD$ in the Picard group of~$X$ is denoted~$\class{\SD}$.

For a nonempty open subset $Y\subseteq X$ we denote
\[
\sing{Y} := \bigl\{ \lambda\in \units{K}\st \forall_{\gp\in Y}\ord_\gp\lambda\equiv 0\pmod{2}\bigr\}
\]
and we set $\SingY := \sfrac{\sing{Y}}{\squares{K}}$. If~$Y$ is a proper subset we further define
\[
\dl{Y} := \sing{Y}\cap \bigcap_{\gp\in X\setminus Y} \squares{K_\gp}
\qquad\text{and}\qquad
\Dl{Y} := \sfrac{\dl{Y}}{\squares{K}}.
\]
The square-class group $\sqgd{K_\gp}$ of a local field~$K_\gp$ consists of fours classes
\[
\sqgd{K_\gp} = \{ 1, u_\gp, \pi_\gp, \pi_\gp u_\gp\},
\]
where $\ord_\gp u_\gp\equiv 0\pmod{2}$ and $\ord_\gp \pi_\gp\equiv 1\pmod{2}$. We call~$u_\gp$ the \term{$\gp$-primary unit}.

We may treat the quotient group $\sfrac{\PicX}{2\PicX}$ as a $\FF_2$-vector space. Given a proper nonempty open subset $Y\subsetneq X$ denote by~$G_Y$ a subgroup of spanned by the classes of points not in~$Y$:
\[
G_Y := \lin_{\FF_2}\Bigr\{ \class{\gp}+ 2\PicX\st \gp\in X\setminus Y\Bigr\}.
\]

A pair of maps $\Tt$, where $T:X\iso X$ is a bijection and $t:\sqgd{K}\iso \sqgd{K}$ is a group automorphism, is called a \term{self-equivalence} of~$K$ if it preserves Hilbert symbols in a sense that
\[
(\lambda,\mu)_\gp = (t\lambda, t\mu)_{T\gp}
\qquad\text{for all }\gp\in X\text{ and }\lambda, \mu\in \sqgd{K}.
\]
Given a self-equivalence $\Tt$, a point~$\gp$ is called \term{tame} if $\ord_\gp\lambda\equiv \ord_{T\gp}t\lambda\pmod{2}$ for all $\lambda\in K$. Otherwise $\gp$ is called \term{wild}. The set of all wild points of $\Tt$ is denoted $\wild{\Tt}$. A subset $\CS\subset X$ is called a \term{wild set} if there is a self-equivalence $\Tt$ such that $\CS = \wild{\Tt}$. For brevity, we shall say that a wild set~$\CS$ if \term{of rank~$n$}, if $\rk G_{X\setminus \CS} = n$.

It is well known that a self-equivalence preserves~$-1$ and factors over local squares-classes. Nevertheless, for easy of reference let us write down these two facts explicitly.

\begin{obs}\label{obs:local_squares}
Let $\gp\in X$ and $\Tt$ be a self-equivalence of~$K$. Then~$t$ induces a group isomorphism $\sqgd{K_\gp}\iso \sqgd{K_{T\gp}}$. In particular, if $\lambda$ is an element of~$K$, then $\lambda\in \squares{K_\gp}$ if and only if $t\lambda\in \squares{K_{T\gp}}$.
\end{obs}

\begin{obs}\label{obs:-1to-1}
If $\Tt$ is a self-equivalence, then $t(-1) = -1$.
\end{obs}


\section{$2$-divisibility in a Picard group}\label{sec:2-divisibility}
The problem of divisibility in the Picard group of a curve has been investigated by numerous authors in recent years (let us mention \cite{Sharif13} to give just one example). In~\cite{CKR18} we proved that a singleton $\{\gp\}$ is a wild set if and only if the class of~$\gp$ is $2$-divisible in $\PicX$. Below we investigate the same question for arbitrary divisors. 

Assigning to a divisor~$\SD$ its degree $\deg\SD$, one defines a group epimorphism $\deg: \DivX\onto \ZZ$. It is well known that it factors over equivalence classes of divisors, hence it induces a group homomorphism $\PicX\to \ZZ$. In particular we have

\begin{obs}\label{obs_even_degree_of_even_div} 
If $\class{\SD}\in 2\PicX$, then $\deg\SD$ is even.
\end{obs}

The next proposition is a straightforward generalization of \cite[Proposition~3.6]{CKR18}.

\begin{prop} 
Fix a point $\gp\in X$ of an odd degree and denote $Y := X\setminus \{\gp\}$. Given $\gp_1, \dotsc, \gp_k\in Y$ the following statements about the divisor $\SD := \gp_1 + \dotsb + \gp_k$ are equivalent:
\begin{enumerate}
\item The class of~$\SD$ in $\PicX$ is $2$-divisible.
\item The class of~$\SD$ in $\PicY$ is $2$-divisible and the degree of~$\SD$ is even.
\end{enumerate}
\end{prop}

\begin{proof}
Assume that $\class{\SD}\in 2\PicX$. Then $\deg \SD\in 2\ZZ$ by Observation~\ref{obs_even_degree_of_even_div}. Moreover, by functoriality of $\Pic$, the inclusion $Y\subset X$ induces a group homomorphism $\PicX\to \PicY$. Hence, if we have $\class{\SD} = 2\class{\SE}$ in $\PicX$ for some divisor~$\SE$, then this equality is preserved after we pass to $\PicY$.

Conversely, assume that $\class{\SD}\in 2\PicY$. This means that there is an element $\lambda\in K$ and a divisor $\SE\in \DivY$ such that
\[
\divY \lambda = \SD + 2\SE \in \DivY
\]
Passing from an affine curve~$Y$ to a complete curve~$X$, we have
\[
\divX\lambda = \SD + 2\SE + \ord_\gp(\lambda)\cdot \gp.
\]
In particular the degrees of the divisors on both sides of the equality must be congruent modulo~$2$. Thus, we have
\[
0\equiv \ord_\gp\lambda\cdot \deg\gp\pmod{2}.
\]
Now, the degree of~$\gp$ is odd by the assumption, hence $\ord_\gp \lambda$ must be even. Say $\ord_\gp \lambda = 2k$ for some $k\in \ZZ$. It follows that in $\PicX$ we have $\class{\SD} = 2 \class{\SE + k\cdot \gp}$ and so~$\SD$ is $2$-divisible.
\end{proof}

\begin{lem}\label{lem:lin_dep_iff_subspace} 
Let $Y\neq\emptyset$ be a proper open subset of~$X$. The following two conditions are equivalent:
\begin{enumerate}
\item\label{it:linearly_dependent} The set $\bigl\{\class{\gp} + 2\PicX\st \gp\in X\setminus Y\bigr\}\subset \sfrac{\PicX}{2\PicX}$ is linearly independent over $\FF_2$.
\item\label{it:proper_subspace} $\SingX = \SingY$.
\end{enumerate}
\end{lem}

\begin{proof}
Let $X\setminus Y = \{ \gp_1, \dotsc, \gp_k\}$. Suppose that the classes of $\gp_1, \dotsc, \gp_k$ are linearly dependent in $\sfrac{\PicX}{2\PicX}$, then there are $\alpha_1, \dotsc, \alpha_k\in \FF_2$ not all equal zero, $\lambda\in \un$ and a divisor $\SD$ such that
\[
\alpha_1\cdot \gp_1 + \dotsb + \alpha_k\cdot \gp_k - 2\SD = \divX\lambda.
\]
Thus $\ord_{\gp_i}\lambda$ is odd for at least one index~$i$, hence $\lambda\notin\SingX$. On the other hand, it is clear that~$\lambda$ lies in $\SingY$. It follows that $\SingX\neq \SingY$.

Conversely, assume that $\SingX\neq\SingY$. Since $\SingX\subset\SingY$, this means that there is an element $\lambda\in \SingY$ which is not in $\SingX$. The divisor of~$\lambda$ has a form
\[
\divX\lambda = \ord_{\gp_1}\lambda\cdot \gp_1 + \dotsb + \ord_{\gp_k}\lambda\cdot \gp_k + \sum_{\gq\in Y}\ord_\gq\lambda\cdot \gq.
\]
Now, $\ord_\gq\lambda\equiv 0\pmod{2}$ for every $\gq\in Y$ since $\lambda\in \SingY$. Consequently
\[
\alpha_1\class{\gp_1} + \dotsb + \alpha_k\class{\gp_k}\in 2\PicX,
\]
where $\alpha_i\in \FF_2$ is the reminder modulo~$2$ of $\ord_{\gp_i}\lambda$. Not all of them are equal to zero, because $\lambda\notin \SingX$.
\end{proof}

The next proposition relates the $2$-rank of~$G_Y$ (i.e. the rank of the set $X\setminus Y$) with that of relevant Picard groups.

\begin{prop}\label{prop:2rank_by_GY} 
Let $Y\neq \emptyset$ be a proper open subset of~$X$. Then
\[
\rk\PicY = 1 + \rk\PicZeroX - \rk G_Y.
\]
\end{prop}

\begin{proof}
Denote $k := \rk G_Y$. First assume that $k = 0$. This means that the class of every point~$\gp$ not in~$Y$ is $2$-divisible in $\PicX$. Fix any $\gq\in X\setminus Y$ and set $Z := X\setminus \{\gq\}$. Then \cite[Proposition~2.7]{CKR18} asserts 
that
\begin{equation}\label{eq:CKR_2.7}
\PicZeroX + 1 = \rk \PicZ.
\end{equation}
By functoriality of $\Pic$, the natural inclusion of~$Z$ in~$X$ induces a group homomorphism $\PicX\to \PicZ$. This means that every point $\gp\in Z\setminus Y$, being $2$-divisible in $\PicX$, is $2$-divisible in $\PicZ$. Now, \cite[Lemma~2.5]{CKR18} states that
\[
\rk\PicY = \rk\PicZ
\]
and so the assertion follows from Eq.~\eqref{eq:CKR_2.7}.

Now assume that $k > 0$. Fix a subset $Z\subsetneq X$ containing~$Y$ and such that $\card{X\setminus Z} = k$ and the classes of points not in~$Z$ are linearly independent in $\sfrac{\PicX}{2\PicX}$ and generate the whole group~$G_Y$. Lemma~\ref{lem:lin_dep_iff_subspace} implies that $\SingX = \SingZ$. Therefore, using \cite[Lemma~2.4 and Proposition 2.3.(1)]{CKR18} we obtain
\[
\rk\PicZeroX + 1= \rk \SingX = \rk\SingZ = \rk \PicZ + k.
\]
Now the same argument as in the case $k = 0$, shows that $\rk\PicY = \rk\PicZ$. Consequently
\[
\rk\PicY = 1 + \rk\PicZeroX - k
\]
as desired.
\end{proof}

One consequence of the previous proposition is that every self-equivalence preserves $2$-ranks of the groups we are interested in.

\begin{prop}\label{prop:Tt_preserves_2ranks} 
Let $\Tt$ be a self-equivalence of~$K$. Assume that $\Tt$ is tame on a nonempty open subset $Y\subsetneq X$. Then:
\begin{enumerate}
\item The self-equivalence $\Tt$ preserves the subgroups~$\SingY$ and $\Dl{Y}$ of $\sqgd{K}$ in the sense that:
\[
t(\SingY) = \Sing{TY}\qquad\text{and}\qquad t(\Dl{Y}) = \Dl{TY}.
\]
\item The self-equivalence $\Tt$ preserves the $2$-ranks of $\PicY$ and~$G_Y$ in the sense that:
\[
\rk \PicY = \rk \PicTY\qquad\text\qquad \rk G_Y = \rk G_{TY}.
\]
\end{enumerate}
\end{prop}

\begin{proof}
The self-equivalence $\Tt$ is tame on~$Y$, hence the equality $t(\SingY) = \Sing{TY}$ follows directly from the definition of the group~$\Sing{Y}$. Take any $\lambda\in \Dl{Y} = \SingY \cap \bigcap_{\gp\notin Y} \squares{K_\gp}$. Then $t\lambda\in \Sing{TY}$ by the previous part and~$t$ preserves the local squares by Observation~\ref{obs:local_squares}. It follows that $t\lambda\in \Dl{TY}$.

In order to prove the second assertion, observe that combining \cite[Proposition~2.3.(2)]{CKR18} with the previous point we obtain
\begin{equation}\label{eq:PicY=PicTY}
\rk \PicY = \rk \Dl{Y} = \rk \Dl{TY} = \rk \PicTY.
\end{equation}
Finally, Proposition~\ref{prop:2rank_by_GY} relates the $2$-ranks of $\PicY$ and~$G_Y$. We have
\begin{align*}
\rk G_Y &= 1 + \rk \PicZeroX - \rk \PicY,\\
\rk G_{TY} &= 1 + \rk \PicZeroX - \rk \PicTY.
\end{align*}
The right-hand-sides agree by Eq.~\eqref{eq:PicY=PicTY} and so do the left-hand-sides.
\end{proof}


\section{Pre-equivalence}
In this section we introduce a new method of constructing a self-equivalence with specified properties. First, however, we need to recall a notion of a \term{small equivalence}.

\begin{df}
Let $\emptyset \neq \CS\subset X$ be a finite (hence closed) subset of~$X$ such that $\rk \PicXS = 0$. A triple $\smalleq{\CS}{\gp}$ is called a \term{small $\CS$-equivalence} of~$K$ if
\begin{enumerate}\renewcommand{\theenumi}{SE\arabic{enumi}}
\item $T_{\CS}\colon \CS\to X$ is injective,
\item $t_{\CS}\colon \Sing{X\setminus\CS}\iso \Sing{X\setminus T_\CS\CS}$ is a group isomorphism,
\item for every $\gp\in\CS$ the associated map $t_{\gp}\colon \sqgd{K_{\gp}}\iso \sqgd{K_{T_\CS\gp}}$ is a bijection preserving local squares, i.e. 
\[
t_\gp\bigl(1\cdot \squares{K_\gp}\bigr) = 1\cdot \squares{K_{T_\CS\gp}}
\]
\item\label{it:SE4} the following diagram commutes
\[
\begin{CD}
\Sing{X\setminus\CS} @>i>> \prod\limits_{\gp\in\CS}\sqgd{K_{\gp}}\\
  @VVt_{\CS}V @VV\prod_{\gp\in\CS}t_{\gp}V\\
\Sing{X\setminus T_\CS\CS} @>j>> \prod\limits_{\gp\in\CS}\sqgd{K_{T_\CS\gp}},
\end{CD}
\]
where $i=\prod_{\gp\in\CS} i_\gp$ (resp. $j= \prod_{\gq\in T_\CS\CS}j_\gq$) is a diagonal map constructed from canonical homomorphisms $i_\gp\colon \Sing{X\setminus \CS}\to \sqgd{K_\gp}$ (resp. $j_\gq\colon \Sing{X\setminus T_\CS\CS}\to \sqgd{K_\gq}$).
\end{enumerate}
\end{df}

Every small equivalence $\smalleq{\CS}{\gp}$ extends to a self-equivalence of~$K$ by \cite[Theorem~2 and Lemma~4]{PSCL94}, i.e. there is $\Tt$ such that $\mbox{$T\mid_\CS$} = T_\CS$ and $t\mid_{\Sing{X\setminus\CS}} = t_\CS$. Moreover $\Tt$ is tame outside~$\CS$ and a point $\gp\in\CS$ is a wild point of $\Tt$ if and only if there is $\lambda\in \sqgd{L_\gp}$ such that $\ord_\gp\lambda\not\equiv \ord_{T_\CS\gp}t_\gp\lambda\pmod{2}$. 

We now define a new object, hereafter called a \term{pre-equivalence} that can be used to construct a small-equivalence, hence in turn a self-equivalence of a desired form.

\begin{df}
Let $\emptyset\neq \CS\subset X$ be a finite (hence closed) subset of~$X$. A triple $\preeq{\CS}{\gp}$ is a \term{pre-equivalence} of~$K$, if
\begin{enumerate}\renewcommand{\theenumi}{PE\arabic{enumi}}
\item $\vT\colon \CS\to X$ is injective,
\item $\vt\colon \sfrac{ \sing{X\setminus\CS} }{ \dl{X\setminus\CS} }\iso \sfrac{ \sing{X\setminus \vT\CS} }{ \dl{X\setminus \vT\CS} }$ is a group isomorphism,
\item for every $\gp\in\CS$ the map $\vt_{\gp}\colon \sqgd{K_{\gp}}\iso \sqgd{K_{\vT\gp}}$ is a bijection preserving local squares, i.e. 
\[
\vt_\gp\bigl(1\cdot \squares{K_\gp}\bigr) = 1\cdot \squares{K_{\vT\gp}}
\]
\item\label{it:PE4} the following diagram commutes
\[
\begin{CD}
\sfrac{ \sing{X\setminus\CS} }{ \dl{X\setminus\CS} } @>\check{\imath}>> \prod\limits_{\gp\in\CS}\sqgd{K_{\gp}}\\
@VV\vt{}V @VV\prod_{\gp\in\CS}\vt_{\gp}V\\
\sfrac{ \sing{X\setminus \vT\CS} }{ \dl{X\setminus \vT\CS} } @>\check{\jmath}>> \prod\limits_{\gp\in\CS}\sqgd{K_{\vT\gp}},
\end{CD}
\]
here~$\check{\imath}$ (resp.~$\check{\jmath}$) is defined as
\[
\check{\imath}\bigl(x\cdot \dl{X\setminus\CS}\bigr) := \bigl(x\cdot \squares{K_\gp}\st \gp\in\CS\bigr),\qquad
\check{\jmath}\bigl(x\cdot \dl{X\setminus\vT\CS}\bigr) := \bigl(x\cdot \squares{K_{\gq}}\st \gq\in\vT\CS\bigr).
\]
\end{enumerate}
\end{df}

\begin{thm}\label{thm:extending_pre_eq} 
Let $\emptyset\neq\CS\subset X$ be a finite subset of~$X$ and let $\preeq{\CS}{\gp}$ be a pre-equivalence. If the $2$-ranks of $G_{X\setminus\CS}$ and $G_{X\setminus \vT\CS}$ coincide, then there is a self-equivalence of~$K$ such that:
\begin{itemize}
\item $T\mid_\CS = \vT$;
\item for every $\lambda\in \Sing{X\setminus\CS}$ one has
\[
(t\lambda)\cdot \dl{X\setminus \vT\CS} = \vt(\lambda\cdot \dl{X\setminus\CS});
\]
\item $\Tt$ is tame outside~$\CS$.
\end{itemize}
\end{thm}

\begin{proof}
Let $n := \card{\CS}$ and $\gp_1,\dotsc, \gp_n$ be all the points of~$\CS$. For an index $i\in \{1,\dotsc, n\}$ denote $\gpp_i := \vT\gp_i$ and set $\CSS := \vT\CS = \{ \gpp_1, \dotsc, \gpp_n\}$. Using \cite[Proposition~2.3.(2)]{CKR18} together with Proposition~\ref{prop:2rank_by_GY}, from the assumption on the $2$-ranks of $G_{X\setminus \CS}$ and $G_{X\setminus \CSS}$ we obtain:
\begin{multline*}
\rk\Dl{X\setminus\CS}
= \rk \PicXS
= 1 + \rk \PicZeroX - \rk G_{X\setminus \CS}\\
= 1 + \rk \PicZeroX - \rk G_{X\setminus \CSS}
= \rk \PicXS[']
= \rk \Dl{X\setminus \CSS}.
\end{multline*}
Denote $m := \rk\Dl{X\setminus \CS}$ and let $\{\lambda_1,\dotsc, \lambda_m\}\subset \Dl{X\setminus\CS}$ and $\{\lambda_1',\dotsc, \lambda_m'\}\subset\Dl{X\setminus\CSS}$ be bases of $\Dl{X\setminus\CS}$ and $\Dl{X\setminus\CSS}$, viewed as $\FF_2$-vector spaces. Further, use \cite[65:18]{OMeara00} to pick points $\gq_1, \dotsc, \gq_m\in X\setminus\CS$ and $\gqq_1, \dotsc, \gqq_m\in X\setminus\CSS$ such that for every~$i$, the element~$\lambda_i$ is a local square at each~$\gq_j$ for $j\neq i$ but not a square at~$\gq_i$ (respectively $\lambda_i'$ if a local square at every~$\gqq_j$ for $j\neq i$ but not a square at~$\gqq_i$):
\[
\lambda_i\in \bigcap_{j\neq i} \squares{K_{\gq_j}},\qquad \lambda_i\notin \squares{K_{\gq_i}}
\qquad\text{and}\qquad
\lambda_i'\in \bigcap_{j\neq i} \squares{K_{\gqq_j}},\qquad \lambda_i\notin \squares{K_{\gqq_i}}
\]
Then, \cite[Lemma~4.1]{CKR18} asserts that the classes of $\gq_1, \dotsc, \gq_m$ are linearly independent (over~$\FF_2$) in $\sfrac{ \PicXS }{ 2\PicXS }$. Analogously the classes of $\gqq_1, \dotsc, \gqq_m$ are linearly independent in $\sfrac{ \PicXS['] }{ 2\PicXS['] }$. Enlarge the sets~$\CS$ and~$\CSS$ by attaching points $\gq_1, \dotsc, \gq_m$ (respectively $\gqq_1, \dotsc, \gqq_m$) to them. Set
\[
\CS_1 := \CS\cup \{\gq_1, \dotsc, \gq_m\}\qquad\text{and}\qquad \CSS_1 := \CSS\cup \{\gqq_1, \dotsc, \gqq_m\}.
\]
We then have
\begin{equation}\label{eq:proof:zero_Pics}
\rk \PicXS[_1] = 0 = \rk \PicXS[_1']
\end{equation}
by \cite[Lemma~2.4]{CKR18}.

We will now construct a small $\CS_1$-equivalence of~$K$. First, we take $T_{\CS_1}\colon \CS_1\to X$ defined by a formula
\[
\begin{cases}
T_{\CS_1}(\gp_i) := \gpp_i, & \text{for }1\leq i\leq n\\
T_{\CS_1}(\gq_j) := \gqq_j, & \text{for }1\leq j\leq m.
\end{cases}
\]
Next we define an isomorphism $t_{\CS_1}\colon \Sing{X\setminus \CS_1}\to \Sing{X\setminus \CSS_1}$. To this end, take a basis $\{\mu_1\dl{X\setminus \CS}, \dotsc, \mu_k\dl{X\setminus\CS}\}$ of $\sfrac{ \sing{X\setminus\CS} }{ \dl{X\setminus\CS} }$. Without loss of generality we may assume that $\mu_1,\dotsc, \mu_k$ are local squares at each $\gq_1, \dotsc, \gq_m$. (If it is not a case and $\mu_i\notin \squares{K_{\gq_j}}$ for some $i,j$, we may just replace $\mu_i$ by $\mu_i\lambda_j$.)  Pick representatives $\mu_1',\dotsc, \mu_k'\in \sing{X\setminus\CSS}$ of the images $\vt(\mu_1\dl{X\setminus\CS}), \dotsc, \vt(\mu_k\dl{X\setminus\CS})$. As above, we can assume that $\mu_1', \dotsc, \mu_k'$ are local squares at $\gqq_1, \dotsc, \gqq_m$. It is clear that the classes of $\mu_1',\dotsc, \mu_k'$ form a basis of $\sfrac{\sing{X\setminus \CS'}}{\dl{X\setminus \CS'}}$. Therefore
\[
\Sing{X\setminus\CS_1} = \Sing{X\setminus\CS} = \Dl{X\setminus\CS}\oplus \lin_{\FF_2}\bigl\{\mu_1, \dotsc, \mu_k\bigr\}
\]
and analogously
\[
\Sing{X\setminus\CS_1'} = \Sing{X\setminus\CS'} = \Dl{X\setminus\CS'}\oplus \lin_{\FF_2}\bigl\{\mu_1', \dotsc, \mu_k'\bigr\}.
\]
We then set
\begin{equation}\label{eq:proof:t_S_1}
\begin{cases}
t_{\CS_1}\bigl(\lambda_i\cdot \squares{K}\bigr) := \lambda_i'\cdot \squares{K}& \text{for }1\leq i\leq m\\
t_{\CS_1}\bigl(\mu_j\cdot \squares{K}\bigr) := \mu_j'\cdot \squares{K}& \text{for }1\leq j\leq k
\end{cases}
\end{equation}
and extend it by linearity onto the whole~$\Sing{X\setminus\CS_1}$.

Finally, we define a collection $(t_{\gp_1}, \dotsc, t_{\gp_n}, t_{\gq_1}, \dotsc, t_{\gq_m})$ of isomorphisms of local square-class groups as follows. For $i\leq n$, we set $t_{\gp_i} := \vt_{\gp_i}$, while for $j\leq m$ we let $t_{\gq_j}$ to be the unique group isomorphism from $\sqgd{K_{\gq_j}} = \{ 1, u_j, \pi_j, u_j\pi_j \}$ to $\sqgd{K_{\gqq_j}} = \{ 1, u_j', \pi_j', u_j'\pi_j' \}$ that sends~$u_j$ to~$u_j'$ and~$\pi_j$ to~$\pi_j'$. Here, as always, $u_j$ (resp.~$u_j'$) is a $\gq_j$-primary unit (resp. $\gqq_j$-primary unit), while~$\pi_j$ (resp.~$\pi_j'$) is a corresponding uniformizer.

It remains to check condition~\eqref{it:SE4}. It suffices to notice that it holds for square classes $\lambda_1\squares{K} , \dotsc, \lambda_m\squares{K}$ and $\mu_1\squares{K}, \dotsc, \mu_k\squares{K}$ by the means of conditions~\eqref{it:PE4} and \eqref{eq:proof:t_S_1}. Now these classes form a basis of $\Sing{X\setminus\CS_1}$, hence the condition~\eqref{it:SE4} holds universally.

It follows that $\smalleq{\CS_1}{\gp}$ is a small $\CS_1$-equivalence and so it extends to a self-equivalence $\Tt$ of~$K$ by \cite[Theorem~2 and Lemma~4]{PSCL94} together with \cite[Remark~4.6]{CKR18}.
\end{proof}


\section{Main Results}
In this section we describe the structure of finite wild subsets of~$X$. We begin with a proposition that shows how to construct bigger wild sets by gluing together smaller ones.

\begin{prop}\label{prop:union} 
Let $\CS_1 := \wild{\Tt[1]}$ and $\CS_2 := T_1^{-1}\bigl(\wild{\Tt[2]}\bigr)$ be finite sets. If~$\CS_1$ and~$\CS_2$ are disjoint, then their union is the wild sets of a composition $(T_2\circ T_1, t_2\circ t_1)$:
\[
\CS_1\cup \CS_2 = \wild{(T_2\circ T_1, t_2\circ t_1)}.
\]
\end{prop}

\begin{proof}
It is clear that $(T_2\circ T_1, t_2\circ t_1)$ is a self-equivalence of~$K$. Take a point $\gp\in \CS_1$ and let $u_\gp\in K$ be a $\gp$-primary unit. Then $T_1\gp\notin \CS_2$ so $\Tt[2]$ is tame at $T_1\gp$. On the other $\ord_\gp u_\gp\equiv 0\pmod{2}$ and \cite[Observation~2.1]{CKR18} asserts that $\ord_{T_1\gp} t_1u_\gp \equiv 1\pmod{2}$. Thus we have
\[
\ord_{(T_2\circ T_1)(\gp)} (t_2\circ t_1)(u_\gp)\equiv \ord_{T_1\gp} t_1u_\gp \equiv 1\pmod{2}.
\]
This shows that $\gp\in \wild{(T_2\circ T_1, t_2\circ t_1)}$.

Next, take a point $\gp\in \CS_2$. Then $\Tt[1]$ is tame at~$\gp$, but $T_1\gp$ is a wild point of $\Tt[2]$. Let $u_{T_1\gp}\in K$ be a $(T_1\gp)$-primary unit and denote $v_\gp := t_1^{-1}(u_{T_1\gp})$. We now have
\[
\ord_\gp v_\gp \equiv \ord_{T_1\gp} u_{T_1\gp}\equiv 0\pmod{2},
\]
but $\Tt[2]$ is wild at $T_1\gp$, hence
\[
\ord_{(T_2\circ T_1)(\gp)}(t_2\circ t_1)(v_\gp) \equiv \ord_{T_2(T_1\gp)} t_2 u_{T_1\gp}\equiv 1\pmod{2}.
\]
Consequently~$\gp$ is a wild point of the composition $(T_2\circ T_1, t_2\circ t_1)$.

The previous two paragraphs show that $\CS_1\cup \CS_2\subseteq \wild{(T_2\circ T_1, t_2\circ t_2)}$. It remains to prove that there are no other wild points. To this end take a point $\gp \in X\setminus (\CS_1\cup \CS_2)$ and any element $\lambda\in K$. The self equivalence $\Tt[1]$ is tame at~$\gp$ and $\Tt[2]$ is tame at $t_1\gp$. It follows that
\[
\ord_{(T_2\circ T_1)(\gp)} (t_2\circ t_1)(\lambda) \equiv \ord_{T_1\gp} t_1\lambda\equiv \ord_\gp\lambda\pmod{2}
\]
and so $(T_2\circ T_1, t_2\circ t_1)$ is tame at~$\gp$.
\end{proof}

The next theorem provides a necessary condition for a finite set to be wild, basically it says that the number of elements of a wild set must be at least twice bigger that its rank. Recall that for an open set $Y\subset X$ we defined the subgroup~$G_Y$ to be the subspace of $\sfrac{\PicX}{2\PicX}$ spanned by classes of points not in~$Y$.

\begin{thm}\label{thm:half_points_lin_indep} 
Let $\Tt$ be a self-equivalence of a global function field~$X$ and $\CS = \wild\Tt$ be its finite wild set. Then
\[
\card{\CS} \geq 2\cdot \rk G_{X\setminus \CS}.
\]
\end{thm}

\begin{proof}
Denote $Z := X\setminus \CS$ and $k := \rk G_Z$. Select a maximal linearly independent (in $\sfrac{\PicX}{2\PicX}$) subset $\{\gp_1, \dotsc, \gp_k\}$ of~$\CS$ and let~$Y := X\setminus \{\gp_1, \dotsc, \gp_k\}$, $Z \subset Y$. Proposition~\ref{prop:2rank_by_GY} asserts that
\[
\rk \PicY = 1 + \rk\PicZeroX - k.
\]
It follows from \cite[Proposition~2.3.(1)]{CKR18} that
\begin{equation}\label{eq:2rank_EE_Z}
\rk\SingZ = \card{\CS} + \rk\PicZeroX + 1 - k.
\end{equation}
Now, \cite[Lemma~4.1]{CKR18} asserts that there are elements $\lambda_1, \dotsc, \lambda_k\in \dl{Z}$, whose square classes are linearly independent in~$\Dl{Z}$ and such that for every $i\leq k$ one has
\[
\lambda_i \notin \squares{K_{\gp_i}} 
\qquad\text{and}\qquad
\lambda_i \in \bigcap_{j\neq i} \squares{K_{\gp_j}}.
\]
From the fact that~$\Dl{Z}$ is a subset of~$\SingX$ we infer that the classes of $\lambda_1, \dotsc, \lambda_k$ are linearly independent in~$\SingX$.

The self-equivalence $\Tt$ is tame on~$Z$, hence $t(\SingZ) = \Sing{TZ}$ and for every $\gq\in Z$ we have $\ord_{T\gq}t\lambda_i \equiv \ord_\gq \lambda_i\equiv 0\pmod{2}$. Therefore the elements $t\lambda_1, \dotsc, t\lambda_k$ lie in~$\Sing{TZ}$ and they clearly remain linearly independent. On the other hand, since $\Tt$ is wild on $X\setminus Z$, we must have $\ord_{T\gp_i} t\lambda_i\equiv 1\pmod{2}$ for every $i\leq k$. In particular $t\lambda_1,\dotsc, t\lambda_k\notin \SingX$. This implies that
\[
\rk \sfrac{\Sing{TZ}}{\SingX} \geq k.
\]
Consequently
\[
\rk\Sing{TZ} = \rk\sfrac{\Sing{TZ}}{\SingX} + \rk\SingX \geq k + 1 + \rk\PicZeroX
\]
by \cite[Lemma~2.4]{CKR18}. Now, $\rk\Sing{TZ} = \rk\SingZ$ and the assertion follows by combining the previous inequality with Eq.~\eqref{eq:2rank_EE_Z}.
\end{proof}

The previous theorem provides a necessary condition for a finite set~$\CS$ to be wild in terms of its rank. In the rest of this section we inductively construct wild sets of all possible finite ranks. The case of rank~$0$ was completely resolved in~\cite{CKR18}. For the sake of completeness let us recall:

\begin{thm}[{\cite[Theorem~4.10]{CKR18}}]\label{thm:rank0}
Let~$K$ be a global function field and~$X$ the associated smooth curve. Assume that $Y\subsetneq X$ is a nonempty open subset and denote $\CS := X\setminus Y$. If $\rk G_Y = 0$, then 
$\CS$ is a wild set.
\end{thm}

Now we turn our attention to sets of rank~$1$. Theorem~\ref{thm:half_points_lin_indep} asserts that very such set must consist of at least two points. We need to consider separately the cases when the wild set in question consists of just two or three points. We do it in Lemmas~\ref{lem:rank1n2} and~\ref{lem:rank1n3}, respectively. Then we deal with the general case in Theorem~\ref{thm:rank1}

\begin{lem}\label{lem:rank1n2} 
Let $\gp, \gq\in X$ be two distinct points. If the following two conditions hold
\[
\rk G_{X\setminus\{\gp,\gq\}} = 1\qquad\text{and}\qquad -1\in \squares{K_\gp}\cap \squares{K_\gq}
\]
then there is a self-equivalence $\Tt$ of~$K$ satisfying
\[
T\bigl(\{\gp,\gq\}\bigr) = \{ \gp, \gq\} = \wild{\Tt}.
\]
In particular $\{\gp, \gq\}$ is a wild set.
\end{lem}

\begin{proof}
Denote $\CS := \{\gp, \gq\}$ and $Y := X\setminus \CS$. The $2$-rank of $G_Y$ is~$1$, hence the classes of~$\gp$ and~$\gq$ in $\PicX$ cannot be simultaneously $2$-divisible. Neither they can be linearly independent over~$\FF_2$ in $\sfrac{\PicX}{2\PicX}$. Thus we are left with only two possibilities: either one of the two classes $\class\gp, \class\gq\in \PicX$ is $2$-divisible while the other is not, or none of them is $2$-divisible but $\class{\gp+\gq}\in 2\PicX$. We shall investigate these two cases independently.

First assume that $\class\gp\notin 2\PicX$ and $\class\gq\in 2\PicX$. Then \cite[Proposition~3.4]{CKR18} asserts that there is an element $\lambda\in \SingX$ which is not a local square at~$\gp$, i.e. $\lambda\in \SingX\setminus\squares{K_\gp}$. On the other hand, from the condition $\gq\in 2\PicX$ we infer that there are: an element $\mu\in \un$ and a divisor $\SD\in \DivX$ such that
\[
\divX\mu = \gq + 2\SD.
\]
Thus~$\mu$ has an even valuation everywhere except at~$\gq$. In particular, we have $\ord_\gp\mu\equiv 0\pmod{2}$ and so either~$\mu$ is a square at~$\gp$ or $\mu\equiv \lambda\pmod{\squares{K_\gp}}$. Without loss of generality we may assume that $\mu\in \squares{K_\gp}$. If it is not the case, we may always replace~$\mu$ by $\mu\cdot \lambda$ and~$\SD$ by $\SD + \sfrac12\cdot \divX\lambda$ (notice that $\divX\lambda\in 2\PicX$ since $\lambda\in \SingX$).

Now $\lambda\in \SingX$ hence it has an even valuation everywhere on~$X$ and, as noticed above, $\mu$ has an even valuation everywhere but in~$\gq$. Thus, for every point $\gr\in X\setminus \{\gq\}$ the Hilbert symbol $(\lambda,\mu)_\gr$ is trivial. Hilbert reciprocity law implies that also $(\lambda,\mu)_\gq = 1$. Therefore~$\lambda$ is a local square at~$\gq$.

It follows from the previous two paragraphs that the quotient space $\sfrac{\SingY}{\Dl{Y}}$ is spanned (over~$\FF_2$) by square classes of~$\lambda$ and~$\mu$. Indeed, \cite[Proposition~2.3]{CKR18} states that $\rk\sfrac{\SingY}{\Dl{Y}} = 2$. Moreover $\lambda\notin \Dl{Y}$ since $\lambda\notin \squares{K_\gp}$ and $\mu\notin \Dl{Y}$ because $\ord_\gq\mu \equiv 1\pmod{2}$.  Finally they are linearly independent as $\lambda\not\equiv \mu\pmod{\squares{K_\gq}}$. Thus, the local square-class groups at~$\gp$ and~$\gq$ are
\[
\sqgd{K_\gp} = \big\{ 1 \equiv \mu, \lambda, \pi_\gp, \lambda\cdot \pi_\gp\bigr\}
\qquad\text{and}\qquad
\sqgd{K_\gq} = \bigl\{ 1\equiv \lambda, u_\gq, \mu, u_\gq\cdot \mu\bigr\},
\]
for some $\gp$-uniformizer~$\pi_\gp$ and $\gq$-primary unit~$u_\gq$. Construct a triple $\bigl(\vT, \vt, (\vt_\gp, \vt_\gq)\bigr)$:
\begin{align*}
\vT :\CS&\to \CS, & \vT(\gp) &:= \gq, & \vT(\gq) &:= \gp,\\
\vt :\sfrac{\SingY}{\Dl{Y}}&\to \sfrac{\SingY}{\Dl{Y}}, & \vt(\lambda) &:= \mu, & \vt(\mu) & := \lambda,\\
\vt_\gp : \sqgd{K_\gp}&\to \sqgd{K_\gq}, & \vt_\gp(\lambda) &:= \mu, & \vt_\gp(\pi_\gp) &:= u_\gq,\\
\vt_\gq : \sqgd{K_\gq}&\to \sqgd{K_\gp}, & \vt_\gq(u_\gq) &:= \pi_\gp, & \vt_\gq(\mu) &:= \lambda,
\end{align*}
It is clear that $\bigl(\vT, \vt, (\vt_\gp, \vt_\gq)\bigr)$ is a pre-equivalence and~$\CS$ is the wild set of the induced self-equivalence.

Now assume that $\class\gp, \class\gq$ are not $2$-divisible in $\PicX$ but $[\gp+\gq]$ is $2$-divisible. The proof runs along similar lines as in the previous case. Let again $\lambda, \mu\in K$ be such that $\lambda\in \SingX\setminus \squares{K_\gp}$ and $\divX\mu = \gp+\gq+2\SD$ for some divisor~$\SD$. Then $\ord_\gp\mu\equiv \ord_\gq\mu\equiv 1 \pmod{2}$, hence the Hilbert symbol $(\lambda, \mu)_\gp$ is non-trivial. Hilbert reciprocity law implies that $(\lambda,\mu)_\gp$ is non-trivial, either. Consequently~$\lambda$ is simultaneously a $\gp$-primary and $\gq$-primary unit. It follows that $\sfrac{\SingY}{\Dl{Y}} = \lin_{\FF_2}\{\lambda, \mu\}$ and the local square-class groups at~$\gp$ and~$\gq$ are
\[
\sqgd{K_\gp} = \big\{ 1, \lambda, \mu, \lambda\cdot \mu\bigr\}
\qquad\text{and}\qquad
\sqgd{K_\gq} = \bigl\{ 1, \lambda, \mu, \lambda\cdot \mu\bigr\}.
\]
Define a pre-equivalence $\bigl(\vT, \vt, (\vt_\gp, \vt_\gq)\bigr)$ as follows. Let $\vT :\CS\to \CS$ be an identity and
\begin{align*}
\vt :\sfrac{\SingY}{\Dl{Y}}&\to \sfrac{\SingY}{\Dl{Y}}, & \vt(\lambda) &:= \mu, & \vt(\mu) & := \lambda,\\
\vt_\gp : \sqgd{K_\gp}&\to \sqgd{K_\gp}, & \vt_\gp(\lambda) &:= \mu, & \vt_\gp(\mu) &:= \lambda,\\
\vt_\gq : \sqgd{K_\gq}&\to \sqgd{K_\gq}, & \vt_\gq(\lambda) &:= \mu, & \vt_\gq(\mu) &:= \lambda.
\end{align*}
Then~$\CS$ is the wild set of the induced self-equivalence.
\end{proof}

Now we must scrutinously investigate the case when a wild set consists of precisely three points.

\begin{lem}\label{lem:rank1n3} 
Let $\gp_1, \gp_2, \gp_3$ be three distinct points. Assume that the following two conditions hold:
\[
\rk G_{X\setminus\{\gp_1,\gp_2,\gp_3\}} = 1
\qquad\text{and}\qquad
-1\in \squares{K_{\gp_1}}\cap \squares{K_{\gp_2}}\cap \squares{K_{\gp_3}}.
\]
Then $\{\gp_1, \gp_2, \gp_3\}$ is a wild set.
\end{lem}

\begin{proof}
Denote the set consisting of the three distinguished points $\gp_1, \gp_2, \gp_3$ by~$\CS$. First suppose that one of the three points is $2$-divisible in $\PicX$, say $\class{\gp_3}\in 2\PicX$. Theorem~\ref{thm:rank0} asserts that there is a self-equivalence $\Tt[1]$ of~$K$ such that $\gp_3$ is its unique wild point. Since $\Tt[1]$ is tame on $X\setminus\CS$, we infer from Proposition~\ref{prop:Tt_preserves_2ranks} that
\[
\rk G_{X\setminus \{T_1\gp_1, T_1\gp_2\}} \leq \rk G_{X\setminus\{T_1\gp_1,T_1\gp_2,T_1\gp_3\}} = \rk G_{X\setminus\{\gp_1,\gp_2,\gp_3\}} = 1.
\]
Moreover, $-1 = t_1(-1) \in \squares{K_{\gp_1}}\cap \squares{K_{\gp_2}}$ by Observations~\ref{obs:local_squares} and~\ref{obs:-1to-1}. Thus, the previous lemma says that there is a self-equivalence $\Tt[2]$ with a wild set $\wild{\Tt[2]} = \{ T_1\gp_1, T_1\gp_2\}$. It follows that $\CS$ is a wild set of the composition $(T_2\circ T_1, t_2\circ t_1)$ by Proposition~\ref{prop:union}.

For the rest of the proof we assume that $\class{\gp_i}\notin 2\PicX$ for $1\leq i\leq 3$. The assumption $\rk G_{X\setminus \CS} =1$ implies that $\class{\gp_i + \gp_j}\in 2\PicX$ for all $i,j$ such that $i\neq j$. Thus, for $1\leq i<j\leq 3$ there is an element $\lambda_{ij}\in K$ and a divisor $\SD_{ij}\in \DivX$ such that
\[
\gp_i + \gp_j + 2\SD_{ij} = \divX\lambda_{ij}.
\]
In particular, $\lambda_{ij}\in \sing{X\setminus \{\gp_i,\gp_j\}}$. Without loss of generality we may assume that $\lambda_{23} = \lambda_{12}\cdot \lambda_{13}$.

The class of the point~$\gp_1$ is not $2$-divisible in $\PicX$, hence \cite[Proposition~3.4]{CKR18} asserts that there is an element $\mu\in \singX$ which is not a local square at~$\gp_1$. It must, therefore, be congruent modulo~$\squares{K_{\gp_1}}$ to a $\gp_1$-primary unit. By Hilbert reciprocity law we then have
\[
1
= \prod_{\gp\in X} (\mu, \lambda_{23})_\gp
= (\mu, \lambda_{23})_{\gp_2}(\mu, \lambda_{23})_{\gp_3}
\]
and so $(\mu, \lambda_{23})_{\gp_2} = (\mu, \lambda_{23})_{\gp_3} = -1$. It follows that
\[
\mu\equiv u_{\gp_i}\pmod{\squares{K_{\gp_i}}}
\qquad\text{for }i\in \{1,2,3\}.
\]
Here, as usual, $u_{\gp_i}$ is a $\gp_i$-primary unit. Replacing~$\lambda_{12}$ by $\mu\lambda_{12}$ (and consequently~$\lambda_{23}$ by~$\mu\lambda_{23}$) if needed, we may assume that $\lambda_{12}$ is not a local square at~$\gp_3$.

Now, \cite[Lemma~2.1]{LW92} asserts that there is a point $\gp_4\in X\setminus\CS$ and an element $\nu\in \sing{X\setminus \{\gp_4\}}$ such that $\ord_{\gp_4}\nu$ is odd and
\[
\nu\equiv 1\pmod{\squares{K_{\gp_1}}},\qquad
\nu\equiv \mu\pmod{\squares{K_{\gp_1}}}.
\]
Let us denote $\CS':= \{\gp_1, \gp_2, \gp_4\}$.

Our next step is to describe local square-class groups $\sqgd{K_{\gp_k}}$ for $k\leq 4$ in terms of $\mu, \nu$ and $\lambda_{i,j}$. It is well known that a square-class group of a non-dyadic local field~$K_\gp$ consists of four square-classes, namely: $1, u_\gp, \pi_\gp$ and $u_\gp\pi_\gp$, where~$u_\gp$ is a $\gp$-primary unit and~$\pi_\gp$ is a $\gp$-uniformizer. From what we have proved so far we infer:
\[
\begin{array}{c|cccc}
k & 1 & u_{\gp_k} & \pi_{\gp_k} & u_{\gp_k}\pi_{\gp_k}\\\hline
1 & 1\equiv \nu & \mu & \lambda_{12}\equiv \nu\lambda_{12} & \mu\lambda_{12}\\
2 & 1 & \mu\equiv\nu & \lambda_{12} & \mu\lambda_{12}\equiv \nu\lambda_{12}
\end{array}
\]
Next we describe $\sqgd{K_{\gp_3}}$. We know that~$\mu$ is a $\gp_3$-primary unit and $\ord_{\gp_3}\lambda_{13}\equiv \ord_{\gp_3}\lambda_{23}$ is odd, while $\ord_{\gp_3}\lambda_{12}$ is even. Moreover, we have $\lambda_{12}\equiv \mu$ (and so also $\lambda_{13}\equiv \mu\lambda_{23}$) modulo $\squares{K_{\gp_3}}$. Thus the third line of the table reads as
\[
\begin{array}{c|cccc}
k & 1 & u_{\gp_k} & \pi_{\gp_k} & u_{\gp_k}\pi_{\gp_k}\\\hline
3 & 1 & \mu\equiv\lambda_{12} & \lambda_{23}\equiv \mu\lambda_{13} & \mu\lambda_{23}\equiv \lambda_{13}
\end{array}
\]
It remains to describe $\sqgd{K_{\gp_4}}$. Now $\nu\in \sing{X\setminus \{\gp_4\}}$ and $\mu \in \singX\subset \sing{X\setminus \{\gp_4\}}$, hence for every point $\gp\in X$, $\gp\neq \gp_4$ we have $\ord_\gp \mu\equiv \ord_\gp\nu\equiv 0\pmod{2}$ and so the Hilbert symbol $(\mu, \nu)_\gp$ vanishes everywhere on $X\setminus\{\gp_4\}$. It follows from Hilbert reciprocity law that $(\mu, \nu)_{\gp_4} = 1$, too. Thus~$\mu$ must be a local square at~$\gp_4$, since $\ord_{\gp_4}\nu$ is odd. In order to identify a $\gp_4$-primary unit, observe that both~$\nu$ and~$\lambda_{12}$ lie in $\sing{X\setminus \CS'}$. Using Hilbert reciprocity law one more time we obtain that
\begin{multline*}
1
= \prod_{\gp\in X} (\nu, \lambda_{12})_\gp
= (\nu, \lambda_{12})_{\gp_1}(\nu, \lambda_{12})_{\gp_2}(\nu, \lambda_{12})_{\gp_4}
\\
= (1, \lambda_{12})_{\gp_1}(\mu, \lambda_{12})_{\gp_2}(\nu, \lambda_{12})_{\gp_4}
= -(\nu, \lambda_{12})_{\gp_4}.
\end{multline*}
It follows that~$\lambda_{12}$ is a $\gp_4$-primary unit and we may now complete the table:
\[
\begin{array}{c|cccc}
k & 1 & u_{\gp_k} & \pi_{\gp_k} & u_{\gp_k}\pi_{\gp_k}\\\hline
4 & 1\equiv \mu & \lambda_{12} & \nu & \nu\lambda_{12}
\end{array}
\]

The quotient groups $\sfrac{\Sing{X\setminus \CS}}{\Dl{X\setminus\CS}}$ and $\sfrac{\Sing{X\setminus \CS'}}{\Dl{X\setminus\CS'}}$ may be treated as $\FF_2$-vector spaces. From the above explicit description of the local square-class groups we infer that the sets
\[
\SB = \{ \mu, \lambda_{12}, \lambda_{23} \}
\qquad\text{and}\qquad
\SB' = \{ \mu, \lambda_{12}, \nu \}
\]
are bases of $\sfrac{\Sing{X\setminus \CS}}{\Dl{X\setminus\CS}}$ and $\sfrac{\Sing{X\setminus \CS'}}{\Dl{X\setminus\CS'}}$, respectively. Thus we may construct a pre-equivalence $\bigl(\vT, \vt, (\vt_{\gp_1}, \vt_{\gp_2}, \vt_{\gp_3})\bigr)$. Let $\vT:\CS\to X$ send
\[
\vT(\gp_1) := \gp_1, \qquad
\vT(\gp_2) := \gp_2, \qquad
\vT(\gp_3) := \gp_4.
\]
Next let the remaining maps
\begin{align*}
\vt : \sfrac{\Sing{X\setminus \CS}}{\Dl{X\setminus\CS}} &\to \sfrac{\Sing{X\setminus \CS'}}{\Dl{X\setminus\CS'}}, &
\vt_{\gp_1} : \sqgd{K_{\gp_1}} &\to \sqgd{K_{\gp_1}},\\\
\vt_{\gp_2} : \sqgd{K_{\gp_2}} &\to \sqgd{K_{\gp_2}}, &
\vt_{\gp_3} : \sqgd{K_{\gp_3}} &\to \sqgd{K_{\gp_4}}
\end{align*}
be the unique group homomorphisms that satisfy the following conditions
\begin{align*}
\vt(\mu) &= \nu\lambda_{12} & \vt(\lambda_{12}) &= \mu & \vt(\lambda_{23}) &= \mu\lambda_{12}\\
\vt_{\gp_1}(\mu) &= \nu\lambda_{12} = \lambda_{12} & \vt_{\gp_1}(\lambda_{12}) &= \mu\\
\vt_{\gp_2}(\mu) &= \nu\lambda_{12} = \mu\lambda_{12} & \vt_{\gp_2}(\lambda_{12}) &= \mu\\
\vt_{\gp_3}(\mu) &= \nu\lambda_{12} & \vt_{\gp_3}(\lambda_{23}) &= \mu\lambda_{12} = \lambda_{12}\\
\end{align*}
Theorem~\ref{thm:extending_pre_eq} asserts that this pre-equivalence extends to a self-equivalence $\Tt$ with a wild set $\wild{\Tt} = \CS$. This concludes the proof.
\end{proof}

\begin{thm}\label{thm:rank1} 
Let~$K$ be a global function field and~$X$ the associated smooth curve. Further let $n\geq 2$, $\gp_1, \dotsc, \gp_n\in X$ be distinct points such that $-1$ is a local square at each of them. Denote $\CS := \{\gp_1, \dotsc, \gp_n\}$. If $\rk G_{X\setminus \CS}\leq 1$, then 
$\CS$ is a wild set.
\end{thm}

\begin{proof}
If $\rk G_Y = 0$, then $\gp_1, \dotsc, \gp_n\in 2\PicX$ and the assertion follows from Theorem~\ref{thm:rank0}. Thus we assume that $\rk G_Y = 1$. Proceed by an induction on~$n$. If $n=2$ or $n = 3$, we use Lemma~\ref{lem:rank1n2} and~\ref{lem:rank1n3}, respectively. Assume that $n\geq 3$. Denote $Z := X\setminus \{\gp_1, \gp_2\}$. We have $G_Z \subseteq G_Y$ and so $\rk G_Z\leq 1$. Lemma~\ref{lem:rank1n2} asserts that there is a self-equivalence $\Tt[1]$ such that
\[
\{\gp_1, \gp_2\} = T_1\bigl(\{\gp_1, \gp_2\}\bigr) = \wild{\Tt[1]}.
\]
The self-equivalence $\Tt[1]$ is tame on~$Z$, therefore
\[
\rk G_{T_1Z} =  \rk G_Z\leq 1
\]
by Proposition~\ref{prop:Tt_preserves_2ranks}. Moreover, $t_1(-1) = -1\in \squares{K_{\gp_i}}$ for every~$i$ by assumption, and~$t_1$ maps local squares to local squares by Observation~\ref{obs:local_squares}. Hence $-1$ is a local square at each $T_1\gp_3, \dotsc, T_1\gp_n$. It follows from the inductive hypothesis that there is a self-equivalence $\Tt[2]$ whose wild set is precisely
\[
\wild{\Tt[2]} = \{ T_1\gp_3, \dotsc, T_1\gp_n\}.
\]
The sets $\CS_1 = \{\gp_1,\gp_2\}$ and $\CS_2 = \{\gp_3, \dots, \gp_n\} = T_1^{-1}\wild{\Tt[2]}$ are trivially disjoint, hence their union is again a wild set by Proposition~\ref{prop:union}.
\end{proof}

Finally we focus our attention on wild sets of arbitrary rank. Recall that in~\cite{CKR18} we introduced a relation~$\smile$ on the set of points, whose classes are $2$-divisible in $\PicX$. Let $\gp, \gq\in X$ by such that $\class{\gp}, \class{\gq}\in 2\PicX$. The points $\gp, \gq$ related, denoted $\gp\smile \gq$, when $\Sing{X\setminus \{\gp\}}\setminus \SingX\subseteq \squares{K_\gq}$. This relation is symmetric by \cite[Lemma~4.3]{CKR18}.

\begin{prop}\label{prop:arbitrary_2rank} 
Let $\gp_1, \dotsc, \gp_m$ and $\gq_1, \dotsc, \gq_n$ be distinct points of~$X$, where $m\leq n$. Assume that the following conditions hold:
\begin{enumerate}
\item the classes of $\gp_1, \dotsc, \gp_m$ are linearly independent in $\sfrac{\PicX}{2\PicX}$;
\item the classes of $\gq_1, \dotsc, \gq_n$ are $2$-divisible in $\PicX$;
\item for every $i\leq m$ we have $-1\in \squares{K_{\gp_i}}$;
\item\label{it:pi_squares} for all $i,j\leq n$, $i\neq j$ we have $\gq_i\smile \gq_j$.
\end{enumerate}
Then $\{\gp_1, \dotsc, \gp_m, \gq_1, \dotsc, \gq_n\}$ is a wild set.
\end{prop}

\begin{proof}
By assumption, $-1$ is a local square at every~$\gp_i$ for $i\leq m$. We claim that it is a local square at every~$\gq_i$, $i\leq n$, as well. Indeed, we have $\class{\gq_i}\in 2\PicX$, hence~$\gq_i$ has an even degree by Observation~\ref{obs_even_degree_of_even_div}. Therefore, $-1$ is a square in the residue field $K(\gq_i)$ and consequently $-1\in \squares{K_{\gq_i}}$, as claimed.

In order to prove the proposition we proceed by induction on~$m$. For $m = 1$ the conclusion follows from Theorem~\ref{thm:rank1}. Assume that $m > 1$ and the assertion holds for $m-1$. We must consider two cases. First assume that $n = m$. Denote $\CS_1 := \{\gp_1, \dotsc, \gp_{m-1}, \gq_1, \dotsc, \gq_{m-1}\}$ and $Y_1 := X\setminus \CS_1$. By the inductive hypothesis there is a self-equivalence $\Tt[1]$ of~$K$ such that~$\CS_1$ is its wild set.

By assumption, $\class{\gq_m}$ is $2$-divisible in $\PicX$. We claim that $\class{T_1\gq_m}$ is $2$-divisible, as well. Indeed, \cite[Lemma~4.1]{CKR18} asserts that there are $\lambda_1, \dots, \lambda_m\in \dl{X\setminus \{\gq_1, \dotsc, \gq_m\}}$ such that
\[
\lambda_i\notin \squares{K_{\gp_i}}
\qquad\text{and}\qquad
\lambda_i\in \bigcap_{j\neq i}\squares{K_{\gp_j}}
\]
for every $i\leq m$. By \cite[Proposition~3.2]{CKR18} there is an element $\pi_m\in \sing{X\setminus\{\gq_m\}}\setminus\singX$ such that $\ord_{\gq_m}\pi_m \equiv1 \pmod{2}$. It follows from condition~\eqref{it:pi_squares} that
\[
\pi_m\in \bigcap_{j\leq m-1} \squares{K_{\gq_j}}.
\]
Define an element~$\Lambda$ by a formula
\[
\Lambda := \pi_m\cdot \bigcap_{ \substack{j\leq m\\ \mathclap{\pi_m\notin \squares{K_{\gp_j}}}} } \lambda_j.
\]
It is clear that~$\Lambda$ is a local square at every~$\gp_i$, $i\leq m$ and every~$\gq_j$, $j\leq m-1$. It follows from Observation~\ref{obs:local_squares} that $t_1\Lambda$ is a local square at $T_1\gp_1, \dotsc, T_1\gp_m$ and $T_1\gq_1, \dotsc, T_1\gq_{m-1}$. In particular it has even valuations at all those points. Moreover, for any point $\gr\in Y_1$, $\gr\neq \gq_m$ the valuation of~$\Lambda$ at~$\gr$ is even and so the valuation of $t_1\Lambda$ is even at $T_1\gr$, since $\Tt[1]$ is tame everywhere on~$Y_1$. On the other hand, $\Lambda$ has an odd valuation at~$\gq_m$ and so $t_1\Lambda$ has an odd valuation at $T_1\gq_m$. Summarizing, we have shown
\[
t_1\Lambda\in \Sing{X\setminus\{T_1\gq_m\}}
\qquad\text{and}\qquad
\ord_{T_1\gq_m}t_1\Lambda\equiv 1\pmod{2}.
\]
Thus \cite[Proposition~3.2]{CKR18} asserts that $\class{T_1\gq_m}$ is $2$-divisible in $\PicX$, as claimed.

From the above discussion we infer that $\rk G_{X\setminus \{T_1\gp_m T_1\gq_m\}}\leq 1$. (In fact it can be shown that this $2$-rank is precisely~$1$.) Hence, by Lemma~\ref{lem:rank1n2}, there is a self-equivalence $\Tt[2]$ with a wild set $\wild{\Tt[2]} = \{T_1\gp_m, T_1\gq_m\}$ and such that $T_2$ maps the set $\{T_1\gp_m, T_1\gq_m\}$ to itself. Proposition~\ref{prop:union} asserts that the set $\{\gp_1, \dotsc, \gp_m, \gq_1, \dotsc, \gq_m\}$ is the wild set of the composition $(T_2\circ T_1, t_2\circ t_1)$. 

Finally, we consider the case when $n > m$. By the previous part, there is a self-equivalence $\Tt[1]$ such that $\CS_1 := \{\gp_1, \dotsc, \gp_m, \gq_1, \dotsc, \gq_m\}$ is the wild set of $\Tt[1]$. The classes of $\gq_{m+1}, \dotsc, \gq_n$ are $2$-divisible in $\PicX$ and so are the classes of their images $T_1\gq_{m+1}, \dotsc, T_1\gq_n$ by the same argument as above. Thus, Theorem~\ref{thm:rank0} asserts that there is a self-equivalence $\Tt[2]$ with a wild set $\CS_2 := \{T_1\gq_{m+1}, \dotsc, T_1\gq_n\}$. Using Proposition~\ref{prop:union} again, we show that $\CS_1\cup T_1^{-1}\CS_2$ is the wild set of $(T_2\circ T_1, t_2\circ t_1)$.
\end{proof}


\end{document}